\documentclass{article}
\usepackage[utf8]{inputenc}
\usepackage{graphicx}
\usepackage{hyperref}
\usepackage{authblk}
\usepackage[english]{babel}
\usepackage{mathtools}
\usepackage{amsfonts}
\usepackage{amsmath}
\usepackage{amssymb}
\usepackage{amsthm}

\graphicspath{ {./} }

\title{A generalization of Newton's quadrilateral theorem and an elementary proof of Minthorn's quadrilateral theorem}
\author{Rauan Kaldybayev}
\affil{Williams College}
\date{November 2022}

\newcommand{\R}{\mathbb{R}}
\newcommand{\bs}{\backslash}

\newcommand{\dq}[1]{\lq\lq #1\rq\rq}
\newcommand{\lr}[1]{\left(#1\right)}
\newcommand{\ol}[1]{\overline{#1}}

\newcommand{\column}[2]{\binom{#1}{#2}}

\newcounter{dummy} \numberwithin{dummy}{section}
\newtheorem{lemma}[dummy]{Lemma}
\newtheorem{proposition}[dummy]{Proposition}
\newtheorem{theorem}[dummy]{Theorem}
\newtheorem{corollary}[dummy]{Corollary}
\newtheorem{conjecture}[dummy]{Conjecture}
\theoremstyle{definition} \newtheorem{definition}[dummy]{Definition}

\providecommand{\keywords}[1]
{
  \small	
  \textbf{\textit{Keywords---}} #1
}

\begin{document}

\nocite{*}

\maketitle

\begin{abstract}
Newton's quadrilateral theorem can be phrased as follows. If $H$ is a circle that is tangent to the four extended sides of a non-parallelogram quadrilateral $Q$, the center of $H$ lies on the Newton line of $Q$. We prove that the theorem remains true if $H$ is an arbitrary hyperbola or ellipse. A quadrilateral can have at most one circle tangent to it but infinitely many ellipses and hyperbolas. We also prove a converse of Newton's theorem, namely that every point on the Newton line, excepting three singular points, is the center of some ellipse or hyperbola tangent to the four extended sides of $Q$. Using the same proof techniques we give an elementary proof of the (lesser known) Minthorn's quadrilateral theorem, which concerns quadrilaterals passing through the four vertices of $Q$. Our proofs are analytic; they rely on linear algebra and affine transformations.
\end{abstract}

\keywords{conic, quadrilateral, tangent, locus, line, nine-point conic}

\section{Introduction} \label{section:introduction}

\begin{figure}[h]
\caption{An illustration of Conjectures \ref{conj:codeparade} and \ref{conj:centers of passing conics} obtained by a Monte Carlo simulation.}
\centering
\includegraphics[width=0.5\textwidth]{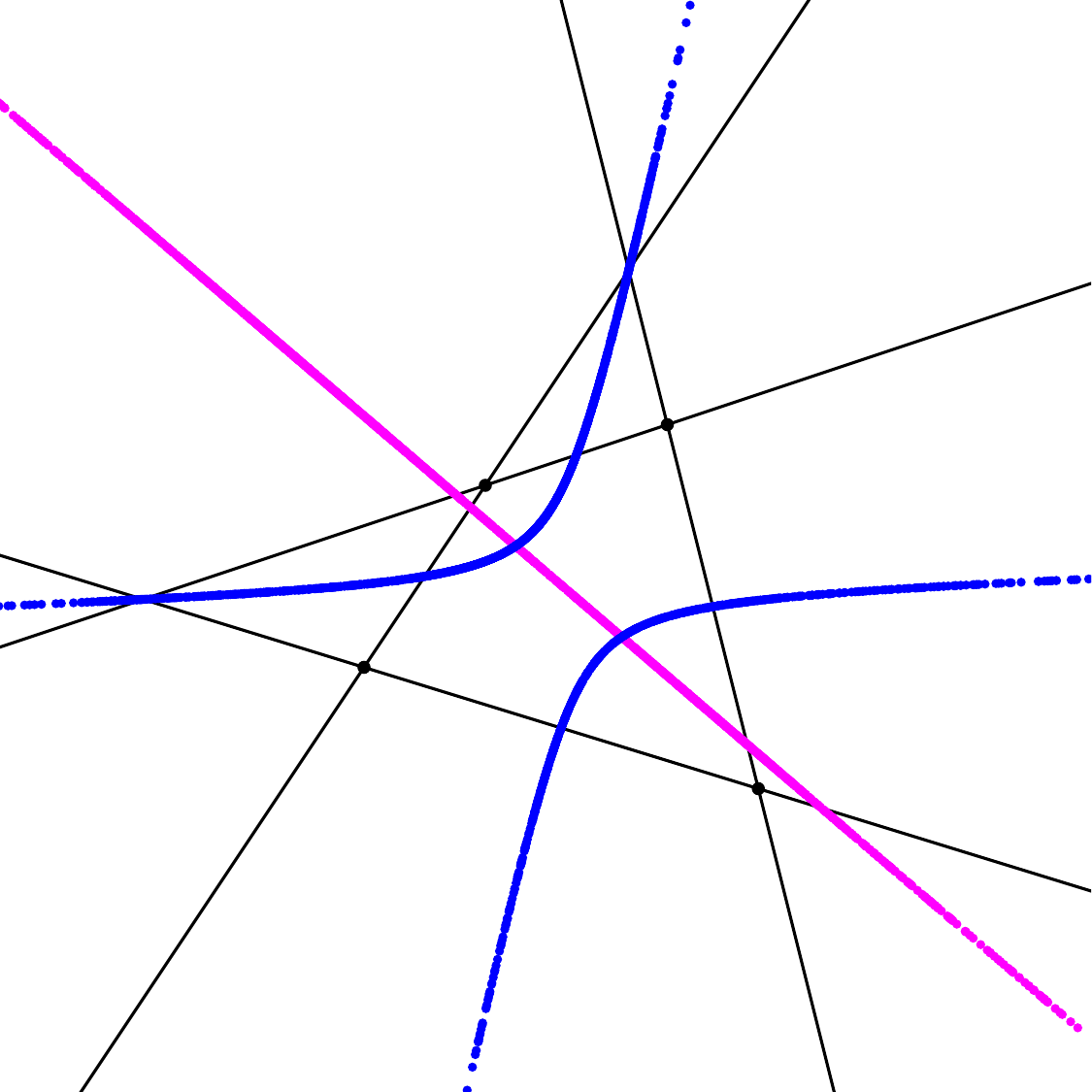} \label{fig:conjs illustration}

\textit{Centers of tangent conics are marked in magenta, centers of passing conics are marked in blue. We fixed a quadrilateral in the plane, randomly generated 3000 tangent conics and 3000 passing conics, and marked the centers of the conics in appropriate colors. Figure \ref{fig:conjs illustration}, as well as the other figures in this paper, was produced in Wolfram Mathematica.}
\end{figure}

This paper was inspired by a conjecture from a video \cite{codeparade} that content creator CodeParade uploaded to YouTube. The conjecture can be approximately phrased as follows:

\begin{conjecture} \label{conj:codeparade}
(CodeParade) If $Q$ is a quadrilateral, the set of centers of all conics tangent to the four extended sides of $Q$ is a straight line.
\end{conjecture}

We make this claim precise and prove it as Theorem \ref{thm:theorem 1}. To the best of our knowledge, the current paper is the first to prove Conjecture \ref{conj:codeparade}. Theorem \ref{thm:theorem 1} can be viewed as a generalization of an old theorem of Newton \cite[p.~117--118]{charmingproofs}. Inspired by Conjecture \ref{conj:codeparade} is Conjecture \ref{conj:centers of passing conics}:

\begin{conjecture} \label{conj:centers of passing conics}
If $Q$ is a convex quadrilateral, the set of centers of all conics passing through the four vertices of $Q$ is a hyperbola; if $Q$ is not convex, this set is an ellipse.
\end{conjecture}

We make this claim precise, make it stronger, and prove it as Theorem \ref{thm:theorem 2}. Conjecture \ref{conj:centers of passing conics} was proved in 1912 by  Maud Minthorn \cite{minthorn}, but we give in this paper a much shorter and more elementary proof. Conjectures \ref{conj:codeparade} and \ref{conj:centers of passing conics} are illustrated in Figure \ref{fig:conjs illustration}.

Conjectures \ref{conj:codeparade} and \ref{conj:centers of passing conics} might be dual to each other, but the exact nature of this duality (if it exists) remains a mystery.

In this paper, we do not use complex numbers or projective geometry; instead, we use real numbers and linear algebra. Our proofs are analytic. Despite the elementary nature of Conjectures \ref{conj:codeparade} and \ref{conj:centers of passing conics}, the proofs of Theorems \ref{thm:theorem 1} and \ref{thm:theorem 2} are surprisingly unenlightening - consisting of dry formal manipulations, they offer very little insight as to why Theorems \ref{thm:theorem 1} and \ref{thm:theorem 2} are true.

We identify the Euclidean plane with $\R^2$; for instance, $(4, \sqrt{7}/2)$ is a point in the plane. Geometric figures are viewed as subsets of $\R^2$. Since $\R^2$ is a vector space, we in some sense treat points as vectors. For instance, if $x$ and $y$ are points in the plane, $\frac{1}{2}x + \frac{1}{2}y$ is their midpoint. The word \dq{collinear} is therefore ambiguous: $(1, -1)$, $(1, 0)$, and $(1, 1)$ are collinear when viewed as points in the plane but not when viewed as elements of a vector space; so we avoid the word \dq{collinear.} Given $x, y, z \in \R^2$, we say that $x, y, z$ \textit{lie on a line} if there exists a line $L$ that contains $x, y, z$.  Given $x, y \in \R^2$, we say that $x$ and $y$ are \textit{multiples of each other} if one vector can be obtained from the other by multiplying it by a real number. Given $x \in \R^2$, we let $x_1$ be the first component of $x$ and $x_2$ the second component, so that $x = (x_1, x_2)$. We denote by $|x|$ the absolute value of $x$, which is equal to $\sqrt{x_1^2 + x_2^2}$. If $A$ is a matrix, we denote the $i$-th entry of the $j$-th column of $A$ as $A_{i j}$, where indexing starts from $1$, so that $A_{11}$ is the top left entry of $A$. To denote the vector whose components are $2$ and $-5$, we will either write $(2, -5)$ (note the comma) or $\column{2}{-5}$. Though row vectors are used in this paper, we do not write row vectors explicitly and will instead express them as transposes of column vectors. If $f: \R^2 \to \R$ is a differentiable function, we let $\nabla f$ be the gradient of $f$, which we treat as a function from $\R^2$ to $\R^2$. One might argue that $\nabla f(x)$ should be a row vector instead of a column vector, but we say it is a column vector for simplicity.

\section{Preliminaries: affine transformations, lines, conics, and the notion of tangency} \label{section:preliminaries}

\begin{definition} \label{defn:affine trans}
A function $\phi: \R^2 \to \R^2$ is an \textit{affine transformation} if and only if it can be expressed as
$$\phi(x) = A x + b$$
for some invertible matrix $A$ and vector $b$.
\end{definition}

Affine transformations are going to be extremely useful to us in this paper, since they allow us to \dq{bend} the plane to our convenience. They preserve the \dq{essence} of geometric figures while letting us vary the details. Under affine transformations, lines map to lines, conics map to conics, and the topology is preserved.

It is known that conic sections can be described in algebraic terms by quadratic polynomials \cite{basictextbook}. It is also known that certain conics, including ellipses and hyperbolas, have a \dq{center.} We proceed to give a formal definition of a center of a geometric figure and prove a few lemmas about how a center of a conic relates to the conic's algebraic description.

\begin{definition} \label{defn:center}
Let $U$ be any subset of $\R^2$, and let $c$ be a point in $\R^2$. Say that $c$ is a \textit{center} of $U$ if and only if $U$ is reflectionally symmetric with respect to $c$. That is, if and only if whenever $x$ lies in $U$, $2 c - x$ also lies in $U$.
\end{definition}

An ellipse or hyperbola has exactly one center. A parabola has zero centers. Certain degenerate conics, such as the line $\{x \in \R^2: x_1^2 = 0\}$, have infinitely many centers, while others, like the \dq{cross} figure $\{x \in \R^2: x_1 x_1 = 0\}$, have exactly one center. Despite the seeming complexity, Lemma \ref{lemma:conic center} gives a simple algebraic description of centers of conic sections. Lemma \ref{lemma:matrix three vectors zero} will be used in the proof.

\begin{lemma} \label{lemma:matrix three vectors zero}
Let $A$ be a symmetric two-by-two matrix, and let $u, v, w$ be three vectors no two of which are multiples of each other. If $u^T A u = 0$, $v^T A v = 0$, and $w^T A w = 0$, then $A$ is the zero matrix.
\end{lemma}
\begin{proof}
None of $u, v, w$ is zero - if $u$ were zero, it could be written as $0 v$. Therefore, we can write $w$ in the basis $\{u, v\}$ as $w = \alpha u + \beta v$ for some real numbers $\alpha$, $\beta$. Since $w$ is not a multiple of $u$ or $v$, $\alpha \ne 0$ and $\beta \ne 0$. Then
$$0 = w^T A w = \alpha^2 u^T A u + \beta^2 v^T A v + 2 \alpha \beta u^T A v = 2 \alpha \beta u^T A v,$$
and therefore $u^T A v = 0$. Now, let $x$ be any vector in $\R^2$. Write $x$ in the basis $\{u, v\}$ as $x = s u + t v$ for some real numbers $s, t$. Then
$$x^T A x = s^2 \, u^T A u + t^2 \, v^T A v + 2 s t \, u^T A v = 0.$$

Since $A$ is symmetric, it is diagonalizable and has real eigenvalues. If $A$ were nonzero, it would have a nonzero eigenvalue $\lambda$; let $q$ be the corresponding eigenvector, and assume without loss of generality that $q$ is real (e.g., $q_1 \in \R$ and $q_2 \in \R$). Then $q^T A q = \lambda |q|^2 \ne 0$, a contradiction. So $A = 0$.
\end{proof}

\begin{lemma} \label{lemma:conic center}
Suppose $H$ is a conic section described by the equation $f(x) = 0$, where $f: \R^2 \to \R$ is a polynomial of degree $2$. A point $c \in \R^2$ is a center of $H$ if and only if the gradient of $f$ evaluates to zero at $c$.
\end{lemma}
\begin{proof}
Write $f(x)$ as
$$f(x) = x^T A x + v^T x + s$$
for some symmetric matrix $A$, vector $v$, and real number $s$. The gradient of $f$ at $c$ is $\nabla f(c) = 2 A c + v$. It is easy to verify that

\begin{equation} \label{eq:f in terms of x - c}
	f(x) = (x - c)^T A (x - c) + (\nabla f(c))^T (x - c) + f(c).
\end{equation}

Therefore,
\begin{equation} \label{eq:f reflection around c}
	f(2 c - x) = f(x) - 2 (\nabla f(c))^T (x - c).
\end{equation}

Suppose $\nabla f(c) = 0$. Equation \ref{eq:f reflection around c} then simplifies to $f(2 c - x) = f(x)$. If $x \in H$, then $f(x) = 0$, so $f(2 c - x) = 0$, so $2 c - x \in H$. By Definition \ref{defn:center}, $c$ is a center of $H$.

Suppose $\nabla f(c) \ne 0$, and suppose for the sake of contradiction that $c$ is nevertheless a center of $H$. If $x$ is a point in $H$, $f(x) = 0$; since $2 c - x$ is also in $H$, $f(2 c - x) = 0$. By equation \ref{eq:f reflection around c}, $(\nabla f(c))^T (x - c) = 0$ for every $x \in H$. The last statement is equivalent to the set containment
$$H \subseteq L, \;\;\; \text{where} \;\;\; L \coloneqq \{x \in \R^2: (\nabla f(c))^T (x - c) = 0\}.$$

Qualitatively, $L$ is the line passing through $c$ whose normal vector is $\nabla f(c)$. Let $y$ be a point in $H$, and let $\ol{y} = 2 c - y$ be the mirror image of $y$ with respect to $c$ (in this paper, we do not consider the empty set a conic, even though it is described by the equation $x^T x + 1 = 0$; so $H$ is guaranteed to have at least one point). Since $\nabla f(c)$ is by assumption nonzero and
$$\nabla f(c) = 2 A c + v = \frac{(2 A y + v) + (2 A (2 c - y) + v)}{2} = \frac{\nabla f(y) + \nabla f(\ol{y})}{2},$$
it has to be that at least one of $\nabla f(y)$ and $\nabla f(\ol{y})$ is nonzero. Assume without loss of generality that $\nabla f(y) \ne 0$.

Both curves $H$ and $L$ pass through the point $y$, and $H$ is contained in $L$. The normal vectors to $H$ and $L$ at $y$ must be multiples of each other - if that were not the case, $H$ would \dq{go at an angle} relative to $L$. Since $\nabla f(y) \ne 0$ and $H$ is the set of points $x$ where $f(x) = 0$, $\nabla f(y)$ is a normal vector to $H$ at the point $y$. By definition, $\nabla f(c)$ is a normal vector to $L$. Therefore,

$$\nabla f(y) = \alpha \: \nabla f(c)$$

for some nonzero real number $\alpha$. Denoting $\delta \coloneqq \nabla f(y)$, we can therefore describe $L$ as the line passing through $y$ whose normal vector is $\delta$:

\begin{equation} \label{eq:L in terms of delta}
	L = \{x \in \R^2: \delta^T(x - y) = 0\}.
\end{equation}

Let $\delta'$ be the result of rotating $\delta$ by $\pi/3$ radians counterclockwise, and let $\delta''$ be the result of rotating $\delta$ by $\pi/3$ radians clockwise. Consider the set $\{\delta, \delta', \delta''\}$. Since $A$ is a symmetric two-by-two matrix, by Lemma \ref{lemma:matrix three vectors zero}, $w^T A w$ being zero for every $w \in \{\delta, \delta', \delta''\}$ would imply $A = 0$, which is a contradiction because $f(x)$ is a polynomial of degree $2$. Let $w \in \{\delta, \delta', \delta''\}$ be such that $w^T A w \ne 0$. Define the function $F: \R \to \R$ as
$$F(t) \coloneqq f(y + t w).$$

We can write $F(t)$ explicitly as
\begin{align*} \begin{split}
	F(t) &= (y + t w)^T A (y + t w) + v^T (y + t w) + s\\
		&= t^2 \cdot w^T A w + t \cdot (2 A y + v)^T w + (y^T A y + v^T y + s)\\
		&= t^2 \cdot w^T A w + t \cdot \delta^T w.
\end{split} \end{align*}

Here, we use that $f(y) = 0$ and that $\nabla f(y) = 2 A y + v$. Because of how $\delta'$ and $\delta''$ were defined, we are guaranteed to have $\delta^T w > 0$. One can check that $t^* = -\frac{\delta^T w}{w^T A w} \ne 0$ is a root of $F$, e.g., $F(t^*) = 0$. This implies that the point $y + t^* w$ lies on $H$. But since $\delta^T ((y + t^* w) - y) = t^* \cdot \delta^T w \ne 0$, equation \ref{eq:L in terms of delta} dictates that $y + t^* w$ does not lie on $L$. This is a contradiction because $H \subseteq L$.
\end{proof}

\begin{corollary} \label{corr:unique center}
If the Hessian matrix of $f$ is nonsingular, $H$ has exactly one center.
\end{corollary}
\begin{proof}
The Hessian of $f$ is $2 A$. Since the Hessian is nonsingular, $A$ is nonsingular. The gradient of $f$, which is equal to $\nabla f(x) = 2 A x + v$, evaluates to zero at exactly one point, namely $-\frac{1}{2}A^{-1} v$; so $-\frac{1}{2}A^{-1} v$ is the unique center of $H$.
\end{proof}

\begin{corollary} \label{corr:description of ellipses and hyperbolas}
If $H$ is an ellipse or hyperbola and $c$ is its center, $H$ can be described as $\{x \in \R^2: F(x) = 0\}$, where $F(x) = (x - c)^T B (x - c) - 1$ for some nonsingular symmetric matrix $B$.
\end{corollary}
\begin{proof}

Since ellipses and hyperbolas do not contain their center, $f(c) \ne 0$. Let
$$B \coloneqq -\frac{1}{f(c)}A.$$

Since $H$ is an ellipse or hyperbola, $A$ is nonsingular, and therefore $B$ is nonsingular. By equation \ref{eq:f in terms of x - c},
$$f(x) = (x - c)^T A (x - c) + f(c),$$
where we use that $\nabla f(c) = 0$ because $c$ is the center of $H$. Then
\begin{align*} \begin{split}
	F(x) &= (x - c)^T B (x - c) - 1\\
		&=-\frac{1}{f(c)}\lr{(x - c)^T A (x - c) + f(c)} = -\frac{1}{f(c)}f(x),
\end{split} \end{align*}
and it follows that $F(x) = 0$ if and only if $f(x) = 0$.
\end{proof}

Tangency of geometric figures is a complicated notion. For the sake of brevity we choose in this paper to use a \dq{makeshift} definition of tangency that only applies to lines and conics. We do so with the hope that every systematic notion of tangency would reduce to Definition \ref{defn:tangency} in the special case of lines and conics. Definition \ref{defn:directed along} provides background for Definition \ref{defn:tangency}.

\begin{definition} \label{defn:directed along}
Let $L$ be a line, and let $H$ be a conic. Say that $L$ is \textit{directed along} $H$ if $H$ is a hyperbola and $L$ is parallel to either of the two asymptotes of $H$, or if $H$ is a parabola and $L$ is parallel to the axis of symmetry of $H$.
\end{definition}

\begin{definition} \label{defn:tangency}
Let $L$ be a line, and let $H$ be a conic. Say that $L$ is \textit{tangent to} $H$ if either of the following holds:
\begin{enumerate}
	\item[1.] $L$ intersects $H$ at exactly one point and is not directed along $H$,
	\item[2.] $H$ is a hyperbola and $L$ is one of the two asymptotes of $H$.
\end{enumerate}
\end{definition}

Case 1 of Definition \ref{defn:tangency} corresponds to the \dq{commonsense} definition of tangency, when $L$ \dq{touches} $H$ but does not cross it. Case 2 of Definition \ref{defn:tangency} declares that the asymptotes of a hyperbola are tangent to it. Even though asymptotes never reach the hyperbola they belong to, they come \dq{infinitely close} to it - that is, no straight line can \dq{fit between} a hyperbola and either of its two asymptotes \cite[p.~124]{salmonconics}. Lemma \ref{lemma:hyperbola asymptotes} gives an algebraic description condition for whether or not a line is directed along an ellipse or hyperbola.

\begin{lemma} \label{lemma:hyperbola asymptotes}
Let $L$ be a line, and let $H$ be an ellipse or hyperbola. Parametrize $L$ as $\{u + v t: t \in \R\}$ for some vectors $u$ and $v$, and write $H$ as the set of points $x$ satisfying $(x - c)^T A (x - c) = 1$ for some nonsingular symmetric matrix $A$ and vector $c$. Then, $L$ is directed along $H$ if and only if $v^T A v = 0$.
\end{lemma}
\begin{proof}
If $H$ is an ellipse, both sides of the biconditional \dq{$L$ is directed along $H$ if and only if $v^T A v = 0$} are false. Indeed, the condition for $L$ being directed along $H$ fails automatically because $H$ is not a hyperbola or parabola, and $v^T A v \ne 0$ because $A$ is positive-definite and $v$ is nonzero.

Suppose $H$ is a hyperbola. Then $\det{A} < 0$, so $A$ has one positive eigenvalue and one negative eigenvalue. By diagonalizing $A$, one can produce two vectors $v_{(1)}$ and $v_{(2)}$ such that
$$v_{(1)}^T A v_{(1)} = 0, \;\;\; v_{(2)}^T A v_{(2)} = 0, \;\;\;  v_{(1)}^T A v_{(2)} = 1/2.$$

If $x$ is a multiple of $v_{(1)}$ or $v_{(2)}$, then $x^T A x = 0$. Since $A$ is a nonzero two-by-two symmetric matrix, by Lemma \ref{lemma:matrix three vectors zero}, the converse is also true: if $x^T A x = 0$, then $x$ is a multiple of $v_{(1)}$ or $v_{(2)}$. Define $L_{(1)}$ and $L_{(2)}$ to be lines consisting of multiples of $v_{(1)}$ and $v_{(2)}$, plus $c$:
$$
	L_{(1)} \coloneqq \{c + v_{(1)} t: t \in \R\}, \;\;\;\;\; L_{(2)} \coloneqq \{c + v_{(2)} t: t \in \R\}.
$$

We claim that $L_{(1)}$ and $L_{(2)}$ are the two asymptotes of $H$. Indeed, let $x \in \R^2$. Write $x - c$ in the basis $\{v_{(1)}, v_{(2)}\}$ as $x - c = \alpha \: v_{(1)} + \beta \: v_{(2)}$ for some $\alpha, \beta \in \R$. Then
$$(x - c)^T A (x - c) = \alpha^2 \: v_{(1)}^T A v_{(1)} + \beta^2 \: v_{(2)}^T A v_{(2)} + 2 \alpha \beta \: v_{(1)}^T A v_{(2)} = \alpha \beta.$$

So $x \in H$ if and only if $\alpha \beta = 1$. This allows us to parametrize $H$ as
$$H = \{c + v_{(1)} t + v_{(2)} \frac{1}{t}: t \in \R, t \ne 0\}.$$

As $t \to 0$, $c + v_{(1)} t + v_{(2)} \frac{1}{t} \approx c + v_{(2)} \frac{1}{t}$ gets closer and closer to $L_{(2)}$, and as $t \to \pm \infty$, $c + v_{(1)} t + v_{(2)} \frac{1}{t} \approx c + v_{(1)} t$ gets closer and closer to $L_{(1)}$. So $L_{(1)}$ and $L_{(2)}$ are the two asymptotes of $H$.

The line $L = \{u + v t: t \in \R\}$ is directed along $H$ if and only if it is parallel to $L_{(1)}$ or $L_{(2)}$, which if and only if $v$ is a multiple of $v_{(1)}$ or $v_{(2)}$, which is if and only if $v^T A v = 0$.
\end{proof}

Finally, we are ready to give an algebraic description of tangency, Lemma \ref{lemma:conic line tangency condition}. We initially discovered Lemma \ref{lemma:conic line tangency condition} in the special case of ellipses using Lagrange multipliers - if $n$ is a normal vector to $L$, we computed the point on $H$ that had the greatest dot product with $n$; if this dot product was equal to $b$, we reasoned that $H$ and $L$ must be tangent. That was not a complete proof, but it was simple and clear. Below we give a formal proof of Lemma \ref{lemma:conic line tangency condition}.

\begin{lemma}[Tangency condition for lines and conics] \label{lemma:conic line tangency condition}
Let $L$ be a line, and let $H$ be an ellipse or hyperbola. Write $L$ as the set of points $x$ satisfying $n^T x = b$ for some nonzero vector $n$ and real number $b$, and write $H$ as the set of points $x$ satisfying $(x - c)^T A (x - c) = 1$ for some nonsingular symmetric matrix $A$ and vector $c$. The line $L$ is tangent to the conic $H$ (according to Definition \ref{defn:tangency}) if and only if $n^T A^{-1} n = (b - n^T c)^2$.
\end{lemma}
\begin{proof}
Let $\sigma$ be the counterclockwise 90-degree rotation matrix,
$$\sigma \coloneqq \begin{pmatrix}
0 & -1\\
1 & 0
\end{pmatrix},$$
so that the cross product $x \times y \coloneqq x_1 y_2 - x_2 y_1$ of two vectors $x$ and $y$ could be expressed as $x \times y = y^T \sigma x$. Since $n$ is a normal vector to $L$, the vector $v \coloneqq \sigma n$ points in the direction of $L$. Observe that
\begin{equation} \label{eq:tangency vn^T - n v^T}
v n^T - n v^T = \sigma n n^T + n n^T \sigma = \begin{pmatrix}
0 & -n_1^2 - n_2^2\\
n_1^2 + n_2^2 & 0
\end{pmatrix} = |n|^2 \sigma.
\end{equation}

Also, note the identity
\begin{equation}
\label{eq:tangency sAs}
\sigma^T A \sigma = \begin{pmatrix}
A_{22} & -A_{12}\\
-A_{12} & A_{11}
\end{pmatrix} = (\det{A}) A^{-1},
\end{equation}
where $A_{12} = A_{21}$ because $A$ is symmetric. 

To prove that $L$ is tangent to $H$ if and only if $n^T A^{-1} n = (b - n^T c)^2$, note that either $L$ is directed along $H$ or it is not. W first prove the biconditional in the case when $L$ is not directed along $H$ and then prove it in the case when $L$ is directed along $H$.

When $L$ is not directed along $H$, by Definition \ref{defn:tangency}, $L$ is tangent to $H$ if and only if $L$ intersects $H$ at exactly one point. We can parametrize $L$ as
$$L = \{u + v t: t \in \R^2\},$$
where $u = c + \frac{b - n^T c}{|n|^2} n$ (and $v = \sigma n$, as previously defined). For $t \in \R$, let
$$f(t) \coloneqq (u + v t - c)^T A (u + v t - c) - 1,$$
so that $u + v t$ lies on $H$ if and only if $f(t) = 0$. The number of times $L$ intersects $H$ is precisely the number of solutions to $f(t) = 0$. An explicit formula for $f(t)$ is
$$f(t) = t^2 \cdot v^T A v + t \cdot 2 v^T A (u - c) + \lr{(u - c)^T A (u - c) - 1}.$$

By Lemma \ref{lemma:hyperbola asymptotes}, the fact $L$ is not directed along $H$ implies that $v^T A v \ne 0$; therefore, $f(t)$ is a polynomial of degree $2$. The quadratic equation $f(t) = 0$ has exactly one solution if and only if the discriminant is zero. That is, if
$$D = \lr{2 v^T A (u - c)}^2 - 4 v^T A v \lr{(u - c)^T A (u - c) - 1} = 0.$$

We now embark to find a convenient formula for $D/4$:
\begin{align*} \begin{split}
	\frac{D}{4} &= \lr{v^T A (u - c)}^2 - \lr{(u - c)^T A (u - c) - 1} v^T A v\\
				&= \frac{(b - n^T c)^2}{|n|^4} (v^T A n)^2 - \frac{(b - n^T c)^2}{|n|^4} n^T A n v^T A v + v^T A v\\
				&= \frac{(b - n^T c)^2}{|n|^4} v^T A n v^T A n - \frac{(b - n^T c)^2}{|n|^4} v^T A v n^T A n + v^T A v\\
				&= \frac{(b - n^T c)^2}{|n|^4} v^T A \lr{n v^T - v n^T} A n + v^T A v.
\end{split} \end{align*}

(We use that $u = c + \frac{b - n^T c}{|n|^2} n$.) Using equation \ref{eq:tangency vn^T - n v^T} and $v = \sigma n$, we can rewrite this as
$$\frac{D}{4} = -\frac{(b - n^T c)^2}{|n|^2} n^T \sigma^T A \sigma A n + n^T \sigma^T A \sigma n.$$

Using equation \ref{eq:tangency sAs}, we arrive at the following relation:
$$\frac{D}{4} = \lr{n^T A^{-1} n - (b - n^T c)^2} (\det{A}).$$

Since $\det{A}$ is nonzero, $D/4$ is zero if and only if $n^T A^{-1} n = (b - n^T c)^2$. So
$$\text{$L$ is tangent to $H$} \iff |L \cap H| = 1 \iff \frac{D}{4} = 0 \iff n^T A^{-1} n = (b - n^T c)^2.$$

This completes the proof in the case when $L$ is not directed along $H$. Now suppose $L$ is directed along $H$. By Lemma \ref{lemma:hyperbola asymptotes}, $v^T A v = 0$, so by equation \ref{eq:tangency sAs}, $n^T A^{-1} n = 0$. Now, $H$ is a hyperbola and $L$ is parallel to one of its two asymtotes; let $L^*$ be the asymptote of $H$ to which $L$ is parallel. Since $c$ is the center of $H$, $L^*$ passes through $c$. The line $L$ may or may not pass through $c$. Since $L$ is parallel to $L^*$, it follows that $L$ equals $L^*$ if and only if $c \in L$, which is if and only if $b - n^T c = 0$. So $L$ is tangent to $H$ if and only if $b - n^T c = 0$. Since $n^T A^{-1} n = 0$, $b - n^T c = 0$ if and only if $(b - n^T c)^2 = n^T A^{-1} n$.
\end{proof}
\begin{corollary}
Let $p, q$ be distinct points in the plane. The line passing through $p$ and $q$ is tangent to the unit circle if and only if $(p \times q)^2 = |p - q|^2$.
\end{corollary}
\begin{proof}
The unit circle is the set of points $x$ satisfying $x^T I x = 1$, where $I$ is the identity matrix. The line passing through $p$ and $q$ is the set of points $x$ satisfying $(q - p) \times x = q \times p$; by the identity $(q - p) \times x = x^T \sigma (q - p)$, we have $\sigma(q - p)$ is a normal vector to this line. By Lemma \ref{lemma:conic line tangency condition}, the line and the circle are tangent if and only if $(q \times p)^2 = (\sigma(q - p))^T (\sigma(q - p))$. Since $\sigma^T \sigma = I$, $(\sigma(q - p))^T (\sigma(q - p)) = (q - p)^T (q - p)$. This completes the proof.
\end{proof}

The proof of the following proposition is left to the reader.

\begin{proposition} \label{prop:affine trans properties}
Let $\phi$ be an affine transformation. Let $p, q, r$ be points, let $L$ be a line, and let $H$ be a conic section. Then
\begin{enumerate}
	\item[1.] $\phi$ is invertible and $\phi^{-1}$ is an affine transformation.
	\item[2.] $\phi(L)$ is a line.
	\item[3.] $\phi(H)$ is a conic section of the same kind as $H$. If $c$ is a center of $H$, $\phi(c)$ is a center of $\phi(H)$. If $H$ is a hyperbola and $L_1$, $L_2$ are its asymptotes, $\phi(L_1)$, $\phi(L_2)$ are the asymptotes of $\phi(H)$.
	\item[4.] $\phi(p)$, $\phi(q)$, $\phi(r)$ lie on a line if and only if $p, q, r$ lie on a line; $\phi(r)$ is the midpoint of $\phi(p)$ and $\phi(q)$ if and only if $r$ is the midpoint of $p$ and $q$.
	\item[5.] $\phi(L)$ is tangent to $\phi(H)$ if and only if $L$ is tangent to $H$.
\end{enumerate}
\end{proposition}

\section{Where the magic happens} \label{section:theorems for square}
In this section, we will prove Conjectures \ref{conj:codeparade} and \ref{conj:centers of passing conics} in the special case of quadrilaterals three of whose vertices are $(0, 0)$, $(1, 0)$, and $(0, 1)$; the general case can be reduced to this special case via an affine transformation.

\begin{figure}[h]
\caption{Illustration of Theorem \ref{thm:theorem_1_square}}
\centering
\includegraphics[width=0.9\textwidth]{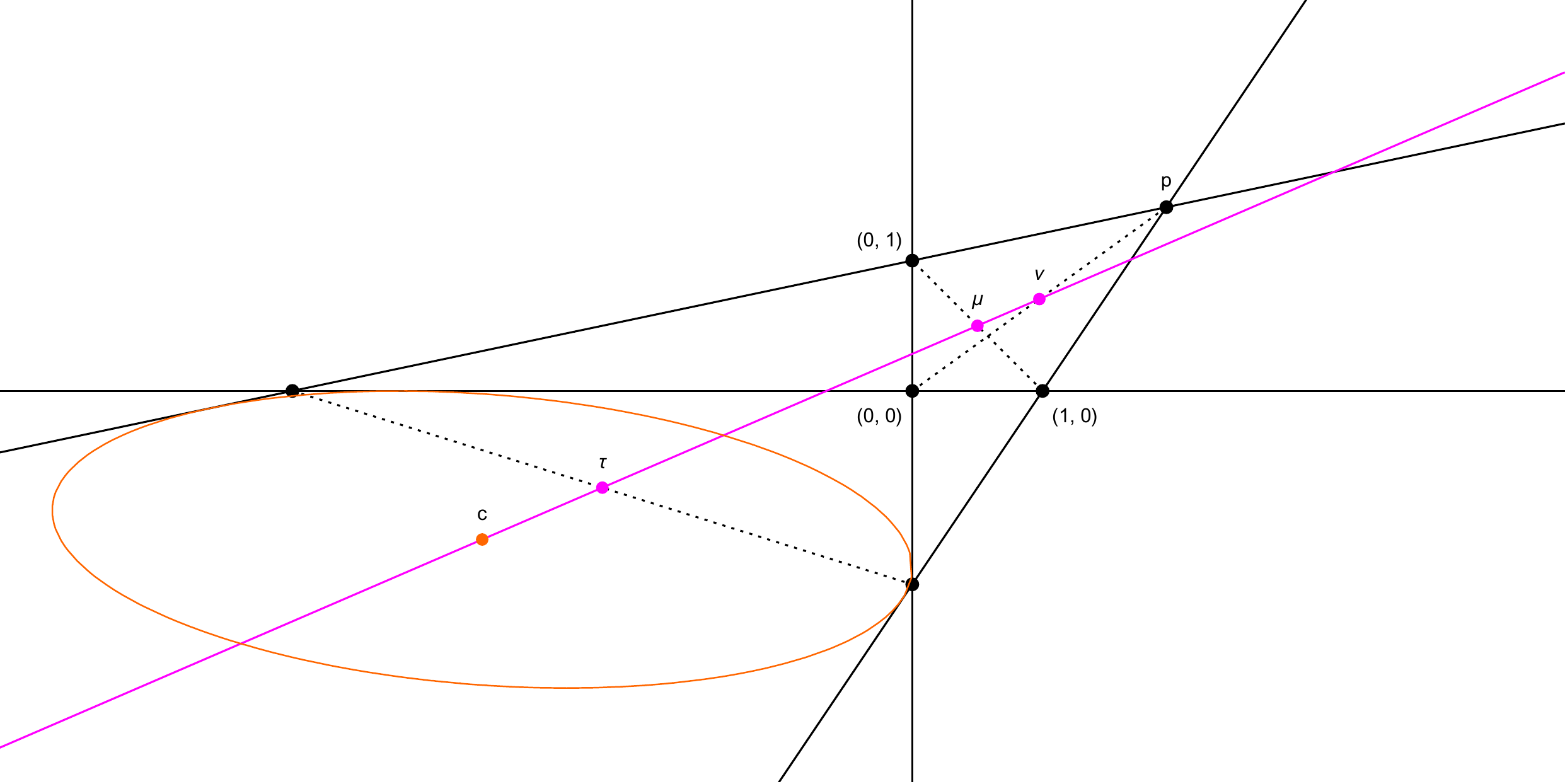} \label{fig:theorem 1 square illustration}
\end{figure}

\begin{theorem} \label{thm:theorem_1_square}
Suppose $Q$ is a quadrilateral (not necessarily simple or convex) whose vertices, listed in order, are the points $(0, 0)$, $(1, 0)$, $p$, $(0, 1)$, where $p \in \R^2$. Suppose also that $Q$ is not a trapezoid and that no three vertices of $Q$ lie on a line. Let $\mu = (1/2, 1/2)$ be the midpoint of the diagonal connecting $(1, 0)$ and $(0, 1)$, and let $\nu = (p_1/2, p_2/2)$ be the midpoint of the diagonal connecting $(0, 0)$ and $p$. Let $\tau = (\frac{1}{2}\frac{p_1}{1 - p_2}, \frac{1}{2}\frac{p_2}{1 - p_1})$ be the midpoint of the line segment connecting the points of intersection of the opposite sides of $Q$ (see the Figure \ref{fig:theorem 1 square illustration}). There exists a line $L$ with the following properties:
\begin{enumerate}
	\item[1.] The points $\mu$, $\nu$, $\tau$ lie on $L$.
	\item[2.] If $c$ is the center of some ellipse or hyperbola tangent to the four extended sides of $Q$, then $c$ lies on $L$.
	\item[3.] Every point of $L$, except the three points $\mu$, $\nu$, $\tau$, is the center of some ellipse or hyperbola tangent to the four extended sides of $Q$.
\end{enumerate}
\end{theorem}
\begin{proof}
Note that since no three points of $Q$ lie on a line, $p_1 \ne 0$ and $p_2 \ne 0$. Since $Q$ is not a trapezoid, $p_1 \ne 1$ and $p_2 \ne 1$.

Since $Q$ is not a trapezoid, it is not a parallelogram, so the midpoints $\mu, \nu$ of its diagonals are distinct. Let $L$ be the line passing through $\mu$ and $\nu$:
\begin{equation} \label{eq:thm_1_square L}
	L \coloneqq \{x \in \R^2: (x - \mu) \times (\nu - \mu) = 0\}.
\end{equation}
Here, $x \times y$  is the cross product of two vectors $x, y \in \R^2$, defined as $x \times y \coloneqq x_1 y_2 - x_2 y_1$. The fact that $\tau$ lies on $L$ can be verified by substituting $\tau$ into equation \ref{eq:thm_1_square L} and referencing the definitions of $\mu$, $\nu$, $\tau$.

An equation of the form $n^T x = b$ gives the line $\{x \in \R^2: n^T x = b\}$. The four extended sides of $Q$ are given by the following four equations:
\begin{equation} \label{eq:thm_1_square horizontal line}
\begin{aligned}[c]
	& \column{1}{0}^T x = 0,\\
	& \column{0}{1}^T x = 0,\\
\end{aligned}
\;\;\;\;\;
\begin{aligned}[c]
	& \column{1 / (2 \tau_1)}{1}^T x = 1,\\
	& \column{1}{1 / (2 \tau_2)}^T x = 1.
\end{aligned}
\end{equation}

Suppose $K$ is an ellipse or hyperbola that is tangent to the four extended sides of $Q$. By Corollary \ref{corr:description of ellipses and hyperbolas}, $K$ can be written as
\begin{equation}
	K = \{x \in \R^2: (x - c)^T A (x - c) = 1\},
\end{equation}
for some nonsingular symmetric matrix $A$ and vector $c$. Clearly, $c$ is the center of $K$. We wish to show that $c$ lies on $L$.

For two distinct points $x, y$ in the plane, denote by $x\#y$ the line that passes through $x$ and $y$. By Lemma \ref{lemma:conic line tangency condition}, the condition that $K$ is tangent to lines $(0, 0)\#(0, 1)$ and $(0, 0)\#(1, 0)$ is
\begin{equation} \label{eq:thm_1_square Ainv 11, 22}
	A^{-1}_{11} = c_1^2, \;\;\; A^{-1}_{22} = c_2^2.
\end{equation}

The condition that $K$ is tangent to lines $(0, 1)\#p$ and $(1, 0)\#p$ can be written as
\begin{align} \begin{split} \label{eq:thm_1_square tangency condition for p lines}
	& (A^{-1}_{22} - c_2^2) + \frac{A^{-1}_{11} - c_1^2}{4 \tau_1^2} + \frac{A^{-1}_{12}}{\tau_1} = 1 - 2 c_2 - \frac{c_1}{\tau_1} + \frac{c_1 c_2}{\tau_1},\\
	& (A^{-1}_{11} - c_1^2) + \frac{A^{-1}_{22} - c_2^2}{4 \tau_2^2} + \frac{A^{-1}_{12}}{\tau_2} = 1 - 2 c_1 - \frac{c_2}{\tau_2} + \frac{c_1 c_2}{\tau_2}.
\end{split} \end{align}

Combining equations \ref{eq:thm_1_square Ainv 11, 22} and \ref{eq:thm_1_square tangency condition for p lines}, we get
\begin{align} \begin{split} \label{eq:thm_1_square Ainv 12}
	& A^{-1}_{12} - c_1 c_2 = \tau_1 - 2 \tau_1 c_2 - c_1,\\
	& A^{-1}_{12} - c_1 c_2 = \tau_2 - 2 \tau_2 c_1 - c_2,
\end{split} \end{align}
and therefore
\begin{equation} \label{eq:thm_1_square lin eq in c}
	\tau_1 - 2 \tau_1 c_2 - c_1 = \tau_2 - 2 \tau_2 c_1 - c_2.
\end{equation}

Look how wonderful! We have obtained a linear equation in $c$. By this point, we are basically done. Note that $(c - \mu) \times (\nu - \mu)$ is a multiple of $(\tau_1 - 2 \tau_1 c_2 - c_1) - (\tau_2 - 2 \tau_2 c_1 - c_2)$. That is, one can verify using the definitions of $\mu$, $\nu$, $\tau$ that for any $x \in \R^2$,
\begin{equation} \label{eq:thm_1_square L eqn in terms of other lin eq in c}
	(x - \mu) \times (\nu - \mu) = -\frac{1}{2} \frac{(p_1 - 1)(p_2 - 1)}{p_1 + p_2 - 1} \big((\tau_1 - 2 \tau_1 x_2 - x_1) - (\tau_2 - 2 \tau_2 x_1 - x_2)\big).
\end{equation}

(Here, $p_1 + p_2 - 1 \ne 0$ because $p$ does not lie on the line connecting $(0, 1)$ and $(1, 0$).) From equations \ref{eq:thm_1_square lin eq in c} and \ref{eq:thm_1_square L eqn in terms of other lin eq in c} it follows that $(c - \mu) \times (\nu - \mu) = 0$, and therefore $c$ lies on $L$.

Note that $c$ cannot be equal to $\mu$. Indeed, if $c = \mu$, equations \ref{eq:thm_1_square Ainv 11, 22} and \ref{eq:thm_1_square Ainv 12} dictate that $A^{-1} = \begin{pmatrix} 1/4 & -1/4\\ -1/4 & 1/4 \end{pmatrix}$. But then $A^{-1}$ fails to be invertible, a contradiction because $A$ is the inverse of $A^{-1}$. A similar argument shows that $c$ cannot be equal to $\nu$ or $\tau$.

Suppose $d$ is a point on $L$ that is not one of $\mu$, $\nu$, and $\tau$. Since $L$ is the line passing through $\mu$ and $\nu$ and $d$ is a point on $L$, we can write $d$ as
\begin{equation} \label{eq:thm_1_square d in terms of t}
	d = \mu + (\nu - \mu) t
\end{equation}
for some real number $t$. Let
\begin{equation} \label{eq:thm_1_square B, k defn}
	B \coloneqq \begin{pmatrix}
		d_1^2 & k\\
		k & d_2^2
	\end{pmatrix}, \;\;\; \text{where} \; k \coloneqq d_1 d_2 + \tau_1 - d_1 - 2 \tau_1 d_2.
\end{equation}

(The definition of $B$ was inspired by equations \ref{eq:thm_1_square Ainv 11, 22} and \ref{eq:thm_1_square Ainv 12}.) If $B$ happens to be invertible, we make the following definition:
\begin{equation} \label{eq:thm_1_square K defn}
	K \coloneqq \{x \in \R^2: (x - d)^T B^{-1} (x - d) = 1\}.
\end{equation}
One can verify using Definition \ref{defn:center} and Lemma \ref{lemma:conic line tangency condition} that $K$ is an ellipse or hyperbola centered at $d$ that is tangent to the four extended sides of $Q$.

Whether $B$ is invertible or not is determined by $d$, which is in turn described by $t$. We therefore wish to express $\det{B} = d_1^2 d_2^2 - k^2$ in terms of $t$. First, consider $k$. Using equations \ref{eq:thm_1_square d in terms of t} and \ref{eq:thm_1_square B, k defn} and expressing $\mu$, $\nu$, $\tau$ in terms of $p$, we arrive at the following surprisingly simple formula:
$$k = d_1 d_2 + \frac{1}{2}(t - 1).$$
With some further algebraic manipulations, one can verify that $\det B$ factors as
$$\det B = -\frac{1}{4} t (t - 1) (t \cdot (p_1 - 1)(p_2 - 1) + (p_1 + p_2 - 1)).$$
The determinant of $B$ is zero if and only if $t = 0$, $t = 1$, or $t = -\frac{p_1 + p_2 - 1}{(p_1 - 1)(p_2 - 1)}$. (Here, $p_1, p_2 \ne 1$ because $Q$ is not a trapezoid.) By equation \ref{eq:thm_1_square d in terms of t}, the values of $d$ corresponding to these three cases are $\mu$, $\nu$, and $\tau$. Since $d$ is not equal to either of $\mu$, $\nu$, and $\tau$, the determinant of $B$ is nonzero, and the conic $K$ as defined in equation \ref{eq:thm_1_square K defn} is an ellipse or hyperbola centered at $d$ and tangent to the four extended sides of $Q$.
\end{proof}

Why is it that $\mu$, $\nu$, and $\tau$ are the only three points of $L$ where no ellipse or hyperbola tangent to the four extended sides of $Q$ can be centered? A somewhat informal explanation that CodeParade alluded to in their YouTube video \cite{codeparade} is that these three points correspond to centers of \dq{infinitely thin ellipses} - that is, ellipses that have \dq{infinitely small minor axis.} These \dq{ellipses} are not formally considered conic sections and we avoid them in our proof, though one might imagine an alternative definition of conics where \dq{infinitely thin ellipses} are considered degenerate conics. We leave this topic and proceed to prove Theorem \ref{thm:theorem_1_trapezoid_but_not_parallelogram}, which can be thought of as expanding Theorem \ref{thm:theorem_1_square} to the case when $Q$ is a trapezoid and $\tau$ is a \dq{point at infinity.}

\begin{theorem} \label{thm:theorem_1_trapezoid_but_not_parallelogram}
Suppose $Q$ is a quadrilateral (not necessarily simple) whose vertices, listed in order, are the points $(0, 0)$, $(1, 0)$, $(1, s)$, $(0, 1)$, where $s \ne 0$ is a real number. Suppose also that $s \ne 1$, so that $Q$ is not a parallelogram. Let $L \coloneqq \{(1/2, t): t \in \R\}$ be the line passing through the midpoints $(1/2, 1/2)$ and $(1/2, s/2)$ of the two diagonals of $Q$.
\begin{enumerate}
	\item[1.] If $c$ is the center of some ellipse or hyperbola tangent to the four extended sides of $Q$, then $c$ lies on $L$.
	\item[2.] Every point of $L$, except $(1/2, 1/2)$ and $(1/2, s/2)$, is the center of some ellipse or hyperbola tangent to the four extended sides of $Q$.
\end{enumerate}
\end{theorem}
\begin{proof}
An equation of the form $n^T x = b$ gives the line $\{x \in \R^2: n^T x = b\}$. The four sides of $Q$ are given by
\begin{equation} \label{eq:thm_1_square_tbnp horizontal line}
\begin{aligned}[c]
	& \column{1}{0}^T x = 0,\\
	& \column{0}{1}^T x = 0,\\
\end{aligned}
\;\;\;\;\;
\begin{aligned}[c]
	& \column{1 - s}{1}^T x = 1,\\
	& \column{1}{0}^T x = 1.
\end{aligned}
\end{equation}

Suppose $K$ is an ellipse or hyperbola that is tangent to the four extended sides of $Q$. Write $K$ as
\begin{equation}
	K = \{x \in \R^2: (x - c)^T A (x - c) = 1\}
\end{equation}
for some nonsingular symmetric matrix $A$ and vector $c$. By Lemma \ref{lemma:conic line tangency condition},
\begin{align} \begin{split} \label{eq:thm_1_square_tbnp Ainv}
	& A^{-1}_{11} = c_1^2,\\
	& A^{-1}_{22} = c_2^2,\\
	& A^{-1}_{22} + 2 (1 - s) A^{-1}_{12} + (1 - s)^2 A^{-1}_{11} = (1 - (1-s) c_1 - c_2)^2,\\
	& A^{-1}_{11} = (1 - c_1)^2.
\end{split} \end{align}

From the first and fourth equations we have $c_1^2 = A^{-1}_{11} = (1 - c_1)^2$, and therefore $c_1 = 1/2$. This shows that $c$ lies on $L$.

If $c$ were equal to $(1/2, 1/2)$, equation \ref{eq:thm_1_square_tbnp Ainv} would dictate that $A^{-1} = \begin{pmatrix} 1/4 & -1/4\\ -1/4 & 1/4\end{pmatrix}$, which is impossible because $A^{-1}$ has to be invertible. A similar argument shows that $c$ cannot be equal to $(1/2, s/2)$.

Let $d$ be a point on $L$ different from $(1/2, 1/2)$ and $(1/2, s/2)$; write $d = (\frac{1}{2}, \frac{1}{2} + \frac{s-1}{2}t)$ for some real number $t$ different from $0$ and $1$. Let
\begin{equation}
	B \coloneqq \begin{pmatrix}
		d_1^2 & m\\
		m & d_2^2
	\end{pmatrix}, \;\;\; \text{where} \; m \coloneqq d_1 d_2 + \frac{1}{2(1-s)} - d_1 - \frac{1}{1-s} d_2.
\end{equation}

The determinant of $B$ factors as
$$\det{B} = -\frac{s}{4} t (t-1).$$

Since $t \ne 0$ and $t \ne 1$, $\det{B} \ne 0$. One can verify using Definition \ref{defn:center} and Lemma \ref{lemma:conic line tangency condition} that
\begin{equation}
	K \coloneqq \{x \in \R^2: (x - d)^T B^{-1} (x - d) = 1\}
\end{equation}
is an ellipse or hyperbola that is centered at $d$ and tangent to the four extended sides of $Q$.
\end{proof}

\begin{figure}[h] \label{fig:theorem 2 square illustration}
\caption{Illustration of Theorem \ref{thm:theorem 2 square}}
\centering
\includegraphics[width=0.9\textwidth]{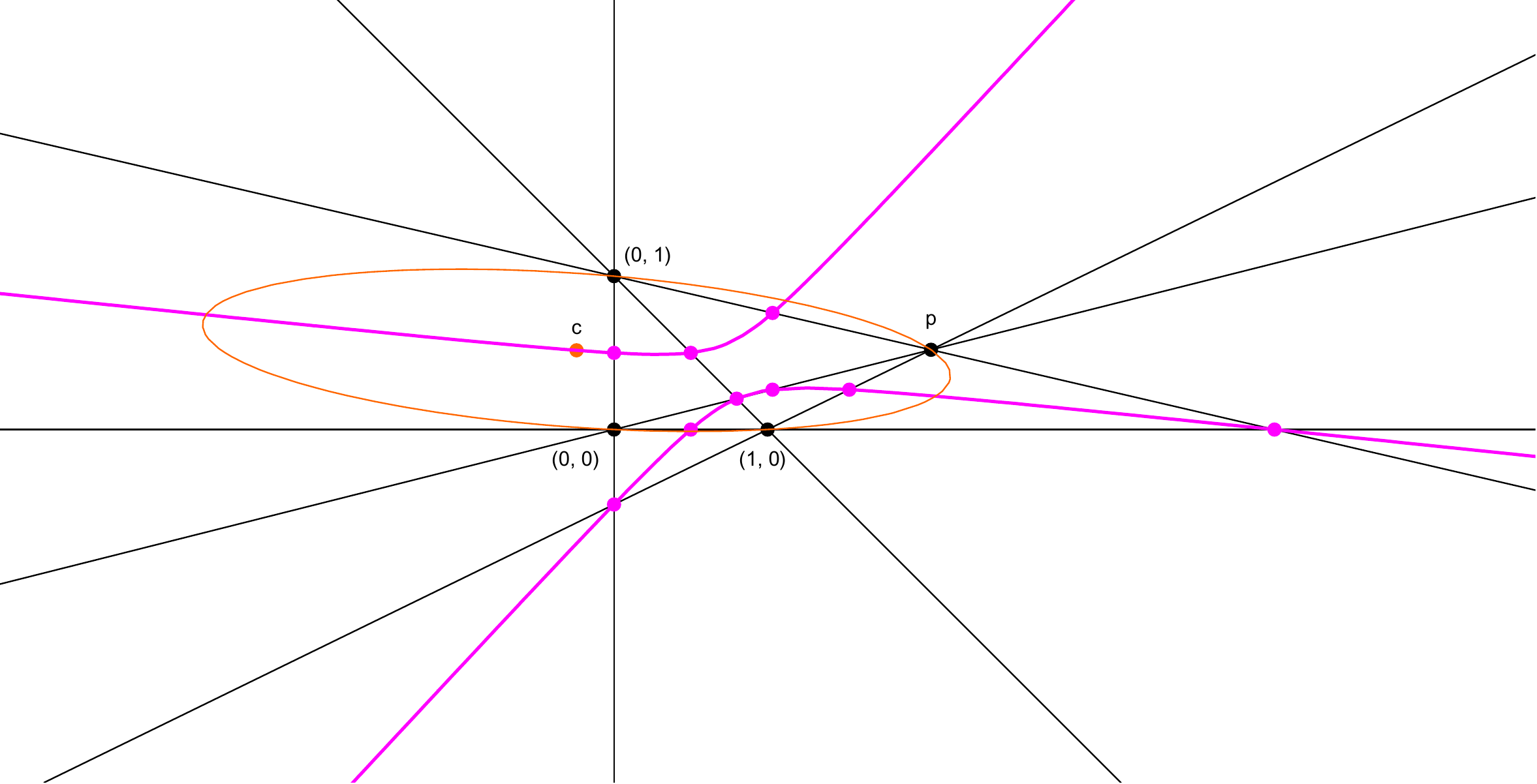}
\end{figure}

We will now shift discussion from conics that are tangent to the four extended sides of a quadrilateral to conics that pass through the four vertices of a quadrilateral. Theorem \ref{thm:theorem 2 square} has been proven by Minthorn \cite{minthorn}, but we give a proof that is shorter and does not employ use the advanced machinery of projective geometry.

\begin{theorem} \label{thm:theorem 2 square}
Suppose $Q$ is a simple quadrilateral whose vertices, listed in order, are the points $(0, 0)$, $(1, 0)$, $p$, $(0, 1)$, where $p \in \R^2$. Suppose also that $Q$ is not a trapezoid, that the diagonals of $Q$ are not parallel, and that no three vertices of $Q$ lie on a line. There exists an ellipse or hyperbola $H$ that satisfies the following properties:
\begin{enumerate}
	\item[1.] Every point of $H$ is the unique center of some conic that passes through the four vertices of $Q$.
	\item[2.] If $c$ is a center of some conic that passes through the four vertices of $Q$, then $c$ lies on $H$.
	\item[3.] $H$ contains the nine points $(\frac{1}{2}, 0)$, $(\frac{1 + p_1}{2}, \frac{p_2}{2})$, $(\frac{p_1}{2}, \frac{1 + p_2}{2})$, $(0, \frac{1}{2})$, $(\frac{1}{2}, \frac{1}{2})$, $(\frac{p_1}{2}, \frac{p_2}{2})$, $(\frac{p_1}{1 - p_2}, 0)$, $(0, \frac{p_2}{1 - p_1})$, $(\frac{p_1}{p_1 + p_2}, \frac{p_2}{p_1 + p_2})$, which are the midpoints of the four sides of $Q$, the midpoints of the two diagonals of $Q$, the two intersections of the (extended) opposite sides of $Q$, and the intersection of the diagonals of $Q$.
	\item[4.] The center of $H$ is $(\frac{1 + p_1}{4}, \frac{1 + p_2}{4})$, which is the arithmetic mean $\frac{1}{4}\lr{\column{0}{0} + \column{1}{0} + \column{0}{1} + p}$ of the four vertices of $Q$.
	\item[5.] $H$ is a hyperbola if and only if $Q$ is strictly convex.
\end{enumerate}
\end{theorem}
\begin{proof}
For two distinct points $x, y$ in the plane, denote by $x\#y$ the line that passes through $x$ and $y$. By assumption, $p$ lies on neither of $(0, 0)\#(0, 1)$, $(0, 0)\#(1, 0)$, $(0, 1)\#(1, 0)$. This yields the following three inequalities:
\begin{equation} \label{eq:thm_2_square p does not lie on lines}
	p_1 \ne 0, \;\;\; p_2 \ne 0, \;\;\; p_1 + p_2 \ne 1.
\end{equation}

Since $Q$ is not a trapezoid, $(0, 0)\#(0, 1)$ is not parallel to $(1, 0)\#p$, and $(0, 0)\#(1, 0)$ is not parallel to $(0, 1)\#p$. Also, the diagonals $(0, 0)\#p$ and $(0, 1)\#(1, 0)$ of $Q$ are not parallel. Therefore,
\begin{equation} \label{eq:thm_2_square no three lines of Q parallel consequences}
	p_1 \ne 1, \;\;\; p_2 \ne 1, \;\;\; p_1 + p_2 \ne 0.
\end{equation}

Define the function $\Gamma: \R^2 \to \R$ as
\begin{multline} \label{eq:thm_2_square Gamma}
	\Gamma(x) \coloneqq p_2(1 - p_2) x_1(x_1 - \frac{1}{2}) + p_1(1 - p_1) x_2 (x_2 - \frac{1}{2}) +
	\\ + 2 p_1 p_2 (x_1 - \frac{1}{2})(x_2 - \frac{1}{2}).
\end{multline}

Though it may look complicated, $\Gamma(x)$ is actually just a quadratic polynomial in $x_1$ and $x_2$. One can check through trivial (albeit laborious) algebra that $\Gamma$ evaluates to zero at the nine points $(0, \frac{1}{2})$, $(\frac{1}{2}, 0)$, $(\frac{1 + p_1}{2}, \frac{p_2}{2})$, $(\frac{p_1}{2}, \frac{1 + p_2}{2})$, $(\frac{1}{2}, \frac{1}{2})$, $(\frac{p_1}{2}, \frac{p_2}{2})$, $(\frac{p_1}{1 - p_2}, 0)$, $(0, \frac{p_2}{1 - p_1})$, $(\frac{p_1}{p_1 + p_2}, \frac{p_2}{p_1 + p_2})$. Hence, we shall call the curve $\{x \in \R^2: \Gamma(x) = 0\}$ the \textit{nine-point conic}.

To determine what kind of shape the nine-point conic is, let us compute the determinant of the Hessian matrix of the function $\Gamma$:
\begin{equation} \label{eq:thm_2_square det H_g}
    \det \text{H}_\Gamma = \det \begin{pmatrix}
        2 p_2(1 - p_2) & 2 p_1 p_2\\
        2 p_1 p_2 & 2 p_1(1 - p_1)
    \end{pmatrix} = -4 p_1 p_2 (p_1 + p_2 - 1).
\end{equation}

Since $p_1, p_2 \ne 0$ and $p_1 + p_2 \ne 1$, $\det \text{H}_\Gamma$ is guaranteed to be nonzero, and by Corollary \ref{corr:unique center} of Lemma \ref{lemma:conic center}, the nine-point conic has a unique center. A simple calculation shows that the gradient of $\Gamma$ evaluates to zero at $(\frac{1 + p_1}{4}, \frac{1 + p_2}{4})$:
$$\nabla \Gamma(\frac{1 + p_1}{4}, \frac{1 + p_2}{4}) = 0.$$

It follows that $(\frac{1 + p_1}{4}, \frac{1 + p_2}{4})$, which is the arithmetic mean of the four vertices of $Q$, is the unique center of $\Gamma$. One can verify that
$$\Gamma(\frac{1 + p_1}{4}, \frac{1 + p_2}{4}) = -\frac{1}{16} (p_1 - 1)(p_2 - 1) (p_1 + p_2).$$

By equation \ref{eq:thm_2_square no three lines of Q parallel consequences}, this is nonzero. Since the nine-point conic does not contain its unique center, it is an ellipse or hyperbola.

Suppose $Q$ is strictly convex. Since $p$ lies to the same side of $(0, 0)\#(0, 1)$ as $(1, 0)$, $p_1 > 0$. Since $p$ lies to the same side of $(0, 0)\#(1, 0)$ as $(0, 1)$, $p_2 > 0$. Since $p$ lies above $(0, 1)\#(1, 0)$, $p_1 + p_2 > 1$. Then $\det \text{H}_\Gamma < 0$ (see equation \ref{eq:thm_2_square det H_g}), and therefore the nine-point conic is a hyperbola.

Suppose $Q$ is not strictly convex, so that one of the four vertices of $Q$ lies inside the triangle formed by the other three vertices (recall that by assumption, no three points of $Q$ lie on a line). First suppose $p$ lies in the triangle formed by $(0, 0)$, $(1, 0)$, $(1, 1)$. Then $p_1 > 0$, $p_2 > 0$, $p_1 + p_2 < 1$, so $\det \text{H}_\Gamma > 0$. Second, suppose $(1, 0)$ lies in the triangle formed by $(0, 0)$, $p$, $(0, 1)$. Then $p_1 > 1$, $p_2 < 0$, $p_1 + p_2 > 1$, so $\det \text{H}_\Gamma > 0$. Third, suppose $(0, 1)$ lies in the triangle formed by $(0, 0)$, $p$, $(1, 0)$. Then $p_1 < 0$, $p_2 > 1$, $p_1 + p_2 > 1$, so $\det \text{H}_\Gamma > 0$. Fourth, suppose $(0, 0)$ lies in the triangle formed by $(1, 0)$, $p$, $(0, 1)$. Then $p_1 < 0$, $p_2 < 0$, $p_1 + p_2 < 1$, so $\det \text{H}_\Gamma > 0$. In either case, $\det \text{H}_\Gamma > 0$, so the nine-point conic is an ellipse.

Having sufficiently investigated the shape of the nine-point conic, we will now proceed to prove how it relates to centers of conics passing through the four vertices of $Q$. We prove part 2 of the current theorem. Suppose $c$ is a center of some conic $K$ that passes through the four vertices of $Q$. Write $K$ as
\begin{equation} \label{eq:thm_2_square 1}
	K = \{x \in \R^2: x^T A x + v^T x + s = 0\},
\end{equation}
where $A$ is a nonsingular symmetric matrix, $v$ is a vector, and $s$ is a real number. We wish to show that $c$ lies on the nine-point conic. That is,
$$
	\Gamma(c) \stackrel{?}{=} 0.
$$

We don't know if $c$ is the unique center of $K$ or if $K$ has many centers; the algebra works out to eventually yield $\Gamma(c) = 0$, so presumably the restrictions we have imposed on $Q$ stipulate that every conic passing through its four vertices either has no center (e.g., is a parabola) or has a unique center. Let
\begin{equation} \label{eq:thm_2_square f defn}
	f(x) \coloneqq x^T A x + v^T x + s.
\end{equation}

The condition that $K$ passes through the four vertices of $Q$ then becomes $f((0, 0)) = 0$, $f((1, 0)) = 0$, $f(p) = 0$, $f((0, 1)) = 0$. From the three equations $f((0, 0)) = 0$, $f((1, 0)) = 0$, $f((0, 1)) = 0$ we obtain the following relations:
\begin{equation} \label{eq:thm_2_square s, A11, A22 formulas}
	s = 0, \;\;\; A_{11} = -v_1, \;\;\; A_{22} = -v_2.
\end{equation}

This allows us to write $f(x)$ as
\begin{align} \begin{split} \label{eq:thm_2_square f formula}
	f(x) =& v^T x + x^T \begin{pmatrix}
-v_1 & \alpha\\
\alpha & -v_2
\end{pmatrix} x\\
	=& v_1 x_1(1 - x_1) + v_2 x_2 (1 - x_2) + 2 \alpha x_1 x_2.
\end{split} \end{align}

Here, $\alpha$ denotes $A_{12}$ (which is also equal to $A_{21}$ because $A$ is symmetric). An additional constraint on $f$ comes from the fact that $K$ must be centered at $c$. The equation produced by Lemma \ref{lemma:conic center} is
\begin{equation} \label{eq:thm_2_square grad f at c}
	\nabla f(c) = 2 A c + v = 0.
\end{equation}

There is one constraint on $f$ that we haven't used yet, namely that $f(p) = 0$. To prove that $\Gamma(c) = 0$, we proceed in cases. First, suppose $\alpha$ is nonzero. The condition $f(p) = 0$ of course implies that $\frac{(c_1 - 1/2)(c_2 - 1/2)}{\alpha} f(p) = 0$. Substituting equation \ref{eq:thm_2_square f formula}, we get
\begin{multline*}
	p_1(1 - p_1) \frac{v_1(c_1 - 1/2)}{\alpha} (c_2 - 1/2) + p_2(1 - p_2) \frac{v_2(c_2 - 1/2)}{\alpha} (c_1 - 1/2) + \\
		+ 2 p_1 p_2 (c_1 - 1/2) (c_2 - 1/2) = 0.
\end{multline*}

From equations \ref{eq:thm_2_square grad f at c} and \ref{eq:thm_2_square s, A11, A22 formulas} it follows that
$$\frac{v_1(c_1 - 1/2)}{\alpha} = c_2, \;\;\; \frac{v_2(c_2 - 1/2)}{\alpha} = c_1.$$

Therefore,
$$
	p_1(1 - p_1) c_2(c_2 - \frac{1}{2}) + p_2(1 - p_2) c_1(c_1 - \frac{1}{2}) + 2 p_1 p_2 (c_1 - \frac{1}{2})(c_2 - \frac{1}{2}) = 0.
$$

We have obtained the desired equality $\Gamma(c) = 0$. (The definition of $\Gamma$ was originally inspired by the above equation.)

Now suppose $\alpha = 0$, so that $A$ is a diagonal matrix. Since $A$ is nonsingular, $A_{11} = -v_1 \ne 0$ and $A_{22} = -v_2 \ne 0$. Equation \ref{eq:thm_2_square grad f at c} simplifies to
$$
	\column{v_1}{v_2} + 2 \begin{pmatrix}
-v_1 & 0\\
0 & -v_2
\end{pmatrix} \column{c_1}{c_2} = \column{0}{0},
$$
and it follows that $c_1 = c_2 = 1/2$, so $c = (1/2, 1/2)$. Direct calculation shows that $\Gamma((1/2, 1/2)) = 0$. This completes the proof of part 2.

To prove part 1 of this theorem, which states that every point on the nine-point conic is the unique center of some conic that passes through the four vertices of $Q$, let $d$ be a point with
$$\Gamma(d) = 0.$$

Let us first work in the assumption that neither of $d_1, d_2$ is equal to $0$ or $1/2$ and that $d_1 + d_2 \ne 1/2$. Define the matrix $B$ as
$$B = \begin{pmatrix}
-d_2 (d_2 - 1/2) & (d_1 - 1/2)(d_2 - 1/2)\\
(d_1 - 1/2)(d_2 - 1/2) & -d_1 (d_1 - 1/2)
\end{pmatrix},$$
and let
$$g(x) \coloneqq (x - d)^T B (x - d) - d^T B d.$$
Let $J$ be the conic $g(x) = 0$. Clearly, $d$ is a center of $J$. To see why $d$ is the unique center of $J$, consider the determinant of $B$:
$$\det B = (d_1 - 1/2)(d_2 - 1/2) \frac{d_1 + d_2 - 1/2}{2}.$$
By assumption, $d_1, d_2 \ne 1/2$ and $d_1 + d_2 \ne 1/2$, so $\det B \ne 0$. By Corollary \ref{corr:unique center}, $J$ has exactly one center. Some simple (though laborious) algebra shows that $g((0, 0)) = 0$, $g((1, 0)) = 0$, $g(p) = 0$, $g((0, 1)) = 0$, which means that $J$ passes through the four vertices of $Q$.

We now proceed to list the \dq{leftover cases,} e.g., when $d$ does not satisfy the assumptions $d_1, d_2 \notin \{0, 1/2\}$ and $d_1 + d_2 \ne 1/2$. Suppose $d_1 = 0$, so that $d = (0, s)$ for some $s \in \R$. The condition $\Gamma(d) = 0$ simplifies to $p_1 ((1 - p_1)s - p_2)(s - 1/2) = 0$, so $s = \frac{p_2}{1 - p_1}$ or $s = 1/2$, so $d = (0, \frac{p_2}{1 - p_1})$ or $d = (0, 1/2)$. Similarly, if $d_2 = 0$, then $d = (\frac{p_1}{1 - p_2}, 0)$ or $d = (1/2, 0)$. Suppose $d_1 = 1/2$, so that $d = (1/2, s)$ for some $s \in \R$. The condition $\Gamma(d) = 0$ simplifies to $p_1(1 - p_1) s (s - 1/2) = 0$, so $d = (1/2, 0)$ or $d = (1/2, 1/2)$. Similarly, if $d_2 = 1/2$, then $d = (0, 1/2)$ or $d = (1/2, 1/2)$. Suppose $d_1 + d_2 = 1/2$, so that $d = (\frac{1 + s}{4}, \frac{1 - s}{4})$ for some $s \in \R$. The condition $\Gamma(d) = 0$ simplifies to $(s^2 - 1)(p_1 + p_2)(1 - (p_1 + p_2)) = 0$. By equations \ref{eq:thm_2_square p does not lie on lines} and \ref{eq:thm_2_square no three lines of Q parallel consequences}, this implies that $s^2 - 1 = 0$, so $s \in \{1, -1\}$, and $d = (1/2, 0)$ or $d = (0, 1/2)$. The \dq{leftover cases} are $d = (0, 1/2)$, $d = (1/2, 0)$, $d = (1/2, 1/2)$, $d = (\frac{p_1}{1 - p_2}, 0)$, $d = (0, \frac{p_2}{1 - p_1})$. Let us handle them.

Define the quadratic functions $g_c$, $g_m$, $g_i$ as
\begin{align*} \begin{split}
	& g_c(x) \coloneqq \frac{\mu + \nu}{2}(x_1 - 1/2)^2 + \frac{\mu - \nu}{2}(x_2 - 1/2)^2 - \frac{\mu}{4},\\
	& g_m(x) \coloneqq p_2(1 - p_2) x_1(x_1 + 2 x_2 - 1) - p_1(p_1 + 2p_2 - 1) x_2(1 - x_2),\\
	& g_i(x) \coloneqq x_2 (p_1 - p_1 x_2 - (1 - p_2) x_1),
\end{split} \end{align*}
where in the definition of $g_c$,
$$\mu \coloneqq (p_2 - 1/2)^2 - (p_1 - 1/2)^2, \;\;\; \nu \coloneqq (p_1 - 1/2)^2 + (p_2 - 1/2)^2 - 1/2.$$

Let $J_c$ be the set of points $x$ where $g_c(x) = 0$; define $J_m$ and $J_i$ in terms of $g_m$ and $g_i$ similarly. It can be checked through trivial (though laborious) algebra that $g_c$, $g_m$, and $g_i$ evaluate to zero at the four vertices of $Q$; so $J_c$, $J_m$, and $J_i$ each pass through the four vertices of $Q$. Similarly one can verify that the gradients of $g_c$, $g_m$, and $g_i$ evaluate to zero at $(1/2, 1/2)$, $(0, 1/2)$, and $(\frac{p_1}{1 - p_2}, 0)$, respectively, so by Lemma \ref{lemma:conic center}, $(1/2, 1/2)$ is a center of $J_c$, $(0, 1/2)$ is a center of $J_m$, and $(\frac{p_1}{1 - p_2}, 0)$ is a center of $J_i$. We proceed to use Corollary \ref{corr:unique center} to show that each of $J_c$, $J_m$, and $J_i$ has exactly one center.

The determinants of the Hessians of $g_c$, $g_m$, and $g_i$ can be factored as
\begin{align*} \begin{split}
    & \det{\text{H}_{g_c}} = -4 p_1 p_2 (p_1 - 1) (p_2 - 1),\\
	& \det{\text{H}_{g_m}} = 4 p_2 (1 - p_2) (p_1 + p_2) (p_1 + p_2 - 1),\\
	& \det{\text{H}_{g_i}} = -(1 - p_2)^2.
\end{split} \end{align*}

These are nonzero by equations \ref{eq:thm_2_square p does not lie on lines} and \ref{eq:thm_2_square no three lines of Q parallel consequences}. The conic $J_i$ is degenerate, but it nevertheless has a unique center.

For each of the points $(1/2, 1/2)$, $(0, 1/2)$, and $(\frac{p_1}{1 - p_2}, 0)$, we have exhibited a conic that passes through the four vertices of $Q$ and has that point as the unique center. Similarly one can exhibit two conics such that one is uniquely centered at $(1/2, 0)$ and the other at $(0, \frac{p_2}{1 - p_1})$.
\end{proof}

\section{Stating the general theorems}
In this section we extend Theorems \ref{thm:theorem_1_square} and \ref{thm:theorem 2 square} to generic quadrilaterals. This is made possible by Lemma \ref{lemma:affine reduction to angle quadrilateral}, which is trivial but critically important.

\begin{lemma} \label{lemma:affine reduction to angle quadrilateral}
Let $Q$ be a quadrilateral such that no three points of $Q$ lie on a line. There exists an affine transformation $\phi$ such that $\phi(Q)$ is a quadrilateral three of whose vertices are $(0, 0)$, $(1, 0)$, $(0, 1)$.
\end{lemma}
\begin{proof}
Label the vertices of $Q$ as $q^{(1)}$, $q^{(2)}$, $q^{(3)}$, $q^{(4)}$. Define the matrix $B$ as
$$B \coloneqq \begin{pmatrix}
q^{(2)}_1 - q^{(1)}_1 & q^{(4)}_1 - q^{(1)}_1\\
q^{(2)}_2 - q^{(1)}_2 & q^{(4)}_2 - q^{(1)}_2
\end{pmatrix}.$$

Qualitatively, $B$ is the matrix whose columns are $q^{(2)} - q^{(1)}$ and $q^{(4)} - q^{(1)}$. Since $q^{(1)}$, $q^{(2)}$, $q^{(4)}$ do not lie on a line, $q^{(2)} - q^{(1)}$ and $q^{(4)} - q^{(1)}$ are not multiples of each other, and therefore $B$ is invertible. Define the affine transformation $\phi$ as
$$\phi(x) \coloneqq B^{-1}(x - q^{(1)}).$$

Let $Q' = \phi(Q)$. One can verify that $\phi(q^{(1)}) = (0, 0)$, $\phi(q^{(2)}) = (1, 0)$, $\phi(q^{(4)}) = (0, 1)$; so three of the vertices of $Q'$ are $(0, 0)$, $(1, 0)$, and $(0, 1)$.
\end{proof}

A \textit{complete quadrilateral} is the figure determined by four lines, no three of which are concurrent, and their six points of intersection \cite{cqlwolfram} \cite[pp. 61--62]{johnson}. We do not use the notion of a complete quadrilateral directly, but it is related to Definition \ref{defn:two hidden vertices, third diag midpoint}.

\begin{definition} \label{defn:two hidden vertices, third diag midpoint}
Let $Q$ be a quadrilateral that is not a trapezoid. Let $L_1$, $L_2$, $L_3$, $L_4$ be the four extended sides of $Q$, in this particular order. If $L_1$ intersects $L_3$ at $p$, and if $L_2$ intersects $L_4$ at $q$, we call $p$ and $q$ the \textit{two hidden vertices} of $Q$ and $\frac{p + q}{2}$ the \textit{third diagonal midpoint} of $Q$.
\end{definition}

Theorem \ref{thm:theorem 1} is the main theorem of this paper. The proof is essentially an application of Lemma \ref{lemma:affine reduction to angle quadrilateral} and Proposition \ref{prop:affine trans properties} to Theorem \ref{thm:theorem_1_square}. Parts 1 and 2 of Theorem \ref{thm:theorem 1} can be proven by applying an affine transformation to Newton's quadrilateral theorem \cite[p.~117--118]{charmingproofs}, but part 3 is, to the best of our knowledge, new.

\begin{theorem}[The locus of the center of a tangent conic is the Newton line] \label{thm:theorem 1}
Suppose $Q$ is a quadrilateral that is not a parallelogram. Suppose also that no three points of $Q$ lie on a line. There exists a line $L$ with the following properties:
\begin{enumerate}
	\item[1.] The midpoints of the two diagonals of $Q$ lie on $L$. If $Q$ is not a trapezoid, the third diagonal midpoint of $Q$ lies on $L$.
	\item[2.] If $c$ is the center of some ellipse or hyperbola tangent to the four extended sides of $Q$, then $c$ lies on $L$.
	\item[3.] Every point of $L$, except the midpoints of the two diagonals of $Q$ and the third midpoint of $Q$, is the center of some ellipse or hyperbola tangent to the four extended sides of $Q$.
\end{enumerate}
\end{theorem}
\begin{proof}
By Lemma \ref{lemma:affine reduction to angle quadrilateral}, there exists an affine transformation $\phi$ such that $\phi(Q)$ is a quadrilateral three of whose vertices are $(0, 0)$, $(1, 0)$, $(0, 1)$. Since affine transformations preserve parallelism and $Q$ is not a parallelogram, $\phi(Q)$ is not a parallelogram.

If $\phi(Q)$ is not a trapezoid, by Theorem \ref{thm:theorem_1_square} there exists a line $L'$ satisfying the desired properties for $\phi(Q)$. An application of Proposition \ref{prop:affine trans properties} completes the proof. We provide the details here but omit them in similar future arguments. Let
$$S \coloneqq \{\text{center of } K: K \text{ is an ellipse or hyperbola tangent to } Q\}.$$

Then
\begin{align*} \begin{split}
    \phi(S) &= \{\phi(\text{center of } K): K \text{ is an ellipse or hyperbola tangent to } Q\}\\
        &= \{\text{center of } \phi(K): K \text{ is an ellipse or hyperbola tangent to } Q\}\\
        &= \{\text{center of } K: \phi^{-1}(K) \text{ is an ellipse or hyperbola tangent to } Q\}\\
        &= \{\text{center of } K: K \text{ is an ellipse or hyperbola tangent to } \phi(Q)\}.
\end{split} \end{align*}

Let $\mu'$, $\nu'$ be the midpoints of the two diagonals of $\phi(Q)$, and let $\tau'$ be the third diagonal midpoint of $\phi(Q)$. By Proposition \ref{prop:affine trans properties}, $\phi^{-1}(\mu')$, $\phi^{-1}(\nu')$ are the midpoints of the two diagonals of $Q$, and $\phi^{-1}(\tau')$ is the third diagonal midpoint of $Q$. By Theorem \ref{thm:theorem_1_square}, $L' \supseteq \{\mu', \nu', \tau'\}$ and $\phi(S) = L' \bs \{\mu', \nu', \tau'\}$, and it of course follows that
\begin{align*} \begin{split}
    & \phi^{-1}(L') \supseteq \{\phi^{-1}(\mu'), \phi^{-1}(\nu'), \phi^{-1}(\tau')\},\\
    & S = \phi^{-1}(L') \bs \{\phi^{-1}(\mu'), \phi^{-1}(\nu'), \phi^{-1}(\tau')\}.
\end{split} \end{align*}

Finally, note that since $L'$ is a line, $\phi^{-1}(L')$ is a line.

If $\phi(Q)$ is a trapezoid but not a parallelogram, by Theorem \ref{thm:theorem_1_trapezoid_but_not_parallelogram} there exists a line $L'$ satisfying the desired properties for $\phi(Q)$. An application of Proposition \ref{prop:affine trans properties}, similar to what is presented above, completes the proof.
\end{proof}

We now proceed to prove Theorem \ref{thm:theorem 2}. A \textit{complete quadrangle} is a set of four points, no three lying on a line, and the six lines which join them \cite{cqawolfram}. We do not use the notion of a complete quadrangle directly, but it is related to Theorem \ref{thm:theorem 2}.

\begin{theorem}[The locus of the center of a passing conic is the nine-point conic] \label{thm:theorem 2}
Let $Q$ be a quadrilateral such that $Q$ is not a trapezoid, the diagonals of $Q$ are not parallel, and no three vertices of $Q$ lie on a line. There exists a conic section $H$ that satisfies the following properties:
\begin{enumerate}
	\item[1.] If $c$ is a center of some conic that passes through the four vertices of $Q$, then $c$ lies on $H$.
	\item[2.] Every point of $H$ is the unique center of some conic that passes through the four vertices of $Q$.
	\item[3.] $H$ contains the midpoints of the four sides of $Q$, the midpoints of the two diagonals of $Q$, the two hidden vertices of $Q$, and the intersection of the two diagonals of $Q$.
	\item[4.] The center of $H$ is the arithmetic mean of the four vertices of $Q$.
	\item[5.] $H$ is a hyperbola if and only if $Q$ is strictly convex.
\end{enumerate}
\end{theorem}
\begin{proof}
By Lemma \ref{lemma:affine reduction to angle quadrilateral}, there exists an affine transformation $\phi$ such that $\phi(Q)$ is a quadrilateral three of whose vertices are $(0, 0)$, $(1, 0)$, $(0, 1)$. Since affine transformations preserve parallelism and $Q$ is not a parallelogram, $\phi(Q)$ is not a parallelogram.

By Theorem \ref{thm:theorem 2 square}, there exists a conic section $H'$ satisfying the desired properties for $\phi(Q)$. An application of Proposition \ref{prop:affine trans properties}, similar to what was given in the proof of Theorem \ref{thm:theorem 1}, completes the proof.
\end{proof}
\begin{corollary}
If $S$ is a line segment connecting any two vertices of $Q$, $H$ contains the midpoint of $L$. If $L_1$ is a line passing through any two vertices of $Q$ and $L_2$ is the line passing through the remaining two vertices, $H$ contains the intersection of $L_1$ and $L_2$.
\end{corollary}

\section{Conclusion}
We have stated and proved in this paper two theorems regarding the locus of the center of a conic that is (a) tangent to the four extended sides of a quadrilateral and (b) passing through the four vertices of a quadrilateral. Theorems \ref{thm:theorem 1} and \ref{thm:theorem 2} are surprising and elegant. Take Theorem \ref{thm:theorem 1}, for example. Why is it be that the locus of the center of a conic tangent to a quadrilateral is the Newton line of that quadrilateral? By default we would expect the locus to be a quadratic curve. Surely there must be some very deep and profound explanation why the locus is actually a line. Is Theorem \ref{thm:theorem 1} a special case of some greater theorem that we were not able to see in this paper? Is there some general principle that makes Theorem \ref{thm:theorem 1} hold? Our bland proof, unfortunately, answers neither of those questions; but it is valuable in that it rigorously establishes the truth of Theorem \ref{thm:theorem 1}, a theorem whose beauty is hard to deny.

The value of Theorems \ref{thm:theorem 1} and \ref{thm:theorem 2} is mostly aesthetic; however, Theorems \ref{thm:theorem 1} and \ref{thm:theorem 2} do have applications in practical problems, particularly those involving conics and quadrilaterals. Consider Theorem \ref{thm:theorem 1}, for example. Since a conic has five degrees of freedom and the tangency condition takes away four, there is one degree of freedom remaining. Theorem \ref{thm:theorem 1} provides an easy parameterization of this degree of freedom, and from the proof of Theorem \ref{thm:theorem_1_square} one can extract the exact formula for the conic. The problem of inscribing the biggest possible ellipse inside a quadrilateral can thus easily be solved. Theorem \ref{thm:theorem 1} can be used in mechanical engineering in determining how far can a cone be inserted into a quadrilateral-shaped opening.

\section{Acknowledgements}
I would first and foremost like to thank Williams College and its generous financial aid program for having me as a student, which is what made the production of this paper possible. I would like to thank Professor Cesar Silva, who introduced me to and showed me the charm of pure mathematics and helped me tremendously in pursuing mathematical research. I would like to thank the Wolfram Fundamental Physics Project, which provided me with invaluable guidance in my first steps as a researcher. I would like to thank Professor Ralph Morrison, Professor Cesar Silva, and Professor Steven Miller of the Williams Math Department for advising me on the formatting and publication of this paper. I am grateful to have access to Wolfram Mathematica, where I performed visualizations and checked algebraic results, which was instrumental in completing this paper. The YouTube content creator CodeParade has played a huge part in the production of this paper by producing Conjecture \ref{conj:codeparade}. Last but not least, I would like to thank my family, my previous olympiad physics teachers Darkhan Shadykul and Margulan Tursynkhan, the National School of Physics and Mathematics in Astana, \dq{Daryn} Center of the Ministry of Education and Science of Kazakhstan, MIT OpenCourseWare, educational content creators on YouTube, and all the other wonderful people and organizations, for giving me the background that I am so lucky to have.

\bibliographystyle{plain}

\bibliography{bibfile.bib}

\end{document}